\documentclass[10pt]{amsart}
\usepackage{epsfig}
\usepackage{graphicx}
\usepackage{amscd}
\usepackage{amsmath}
\usepackage{amsxtra}
\usepackage{amsfonts}
\usepackage{amssymb}

\newtheorem{theorem}{Theorem}[section]
\newtheorem{corollary}[theorem]{Corollary}
\newtheorem{lemma}[theorem]{Lemma}
\newtheorem{proposition}[theorem]{Proposition}

\theoremstyle{definition}
\newtheorem{definition}[theorem]{Definition}
\newtheorem{remark}[theorem]{Remark}

\theoremstyle{remark}

\renewcommand{\theclaim}{\textup{\theclaim}}

\newtheorem*{acknowledgements}{Acknowledgements}

\numberwithin{equation}{section}

\newcommand\restr[2]{{% we make the whole thing an ordinary symbol
  \left.\kern-\nulldelimiterspace % automatically resize the bar with \right
  #1 % the function
  \vphantom{\big|} % pretend it's a little taller at normal size
  \right|_{#2} % this is the delimiter
  }}

\def\openone%{\hbox{\upshape \small1\kern-3.3pt\normalsize1}}

{\mathchoice

{\hbox{\upshape \small1\kern-3.3pt\normalsize1}}

{\hbox{\upshape \small1\kern-3.3pt\normalsize1}}

{\hbox{\upshape \tiny1\kern-2.3pt\SMALL1}}

{\hbox{\upshape \Tiny1\kern-2pt\tiny1}}}

\makeatletter

\newbox\ipbox

\newcommand{\ip}[2]{\left\langle #1\, , \,#2\right\rangle}
\newcommand{\diracb}[1]{\left\langle #1\mathrel{\mathchoice

{\setbox\ipbox=\hbox{$\displaystyle \left\langle\mathstrut
#1\right.$}

\vrule height\ht\ipbox width0.25pt depth\dp\ipbox}

{\setbox\ipbox=\hbox{$\textstyle \left\langle\mathstrut
#1\right.$}

\vrule height\ht\ipbox width0.25pt depth\dp\ipbox}

{\setbox\ipbox=\hbox{$\scriptstyle \left\langle\mathstrut
#1\right.$}

\vrule height\ht\ipbox width0.25pt depth\dp\ipbox}

{\setbox\ipbox=\hbox{$\scriptscriptstyle \left\langle\mathstrut
#1\right.$}

\vrule height\ht\ipbox width0.25pt depth\dp\ipbox}

}\right. }

\newcommand{\dirack}[1]{\left. \mathrel{\mathchoice

{\setbox\ipbox=\hbox{$\displaystyle \left.\mathstrut
#1\right\rangle$}

\vrule height\ht\ipbox width0.25pt depth\dp\ipbox}

{\setbox\ipbox=\hbox{$\textstyle \left.\mathstrut
#1\right\rangle$}

\vrule height\ht\ipbox width0.25pt depth\dp\ipbox}

{\setbox\ipbox=\hbox{$\scriptstyle \left.\mathstrut
#1\right\rangle$}

\vrule height\ht\ipbox width0.25pt depth\dp\ipbox}

{\setbox\ipbox=\hbox{$\scriptscriptstyle \left.\mathstrut
#1\right\rangle$}

\vrule height\ht\ipbox width0.25pt depth\dp\ipbox}

} #1\right\rangle}

\newcommand{\cj}[1]{\overline{#1}}

\newcommand{\bz}{\mathbb{Z}}

\newcommand{\br}{\mathbb{R}}

\newcommand{\bt}{\mathbb{T}}
\newcommand{\bn}{\mathbb{N}}

\newcommand{\nar}[1]{\stackrel{#1}{\rightarrow}}

\def\blfootnote{\xdef\@thefnmark{}\@footnotetext}

\renewcommand{\mod}{\operatorname{mod}}

\input xy
\xyoption{all}
\usepackage{amssymb}

\def\H{\mathcal{H}}

\def\-{^{-1}}

\def\U{\mathcal{U}}

\begin{document}

\title[Spectra of measures and wandering vectors]{Spectra of measures and wandering vectors}
\author{Dorin Ervin Dutkay}
\blfootnote{}
\address{[Dorin Ervin Dutkay] University of Central Florida\\
    Department of Mathematics\\
    4000 Central Florida Blvd.\\
    P.O. Box 161364\\
    Orlando, FL 32816-1364\\
U.S.A.\\} \email{Dorin.Dutkay@ucf.edu}

\author{Palle E.T. Jorgensen}
\address{[Palle E.T. Jorgensen]University of Iowa\\
Department of Mathematics\\
14 MacLean Hall\\
Iowa City, IA 52242-1419\\}\email{jorgen@math.uiowa.edu}

\thanks{}

\begin{abstract} We present a characterization of the sets that appear as Fourier spectra of measures in terms of the existence of a strongly continuous representation of the ambient group that has a wandering vector for the given set.
\end{abstract}
\maketitle \tableofcontents

\section{Introduction}

\begin{definition}\label{def1.1}
For $\lambda\in\br^d$, denote by 
$$e_\lambda(x)=e^{2\pi i\lambda\cdot x},\quad(x\in\br^d)$$
Let $\mu$ be a Borel probability measure on $\br^d$. We say that the measure $\mu$ is {\it spectral} if there exists a set $\Lambda$ in $\br^d$, called a {\it spectrum} of $\mu$, such that the set $\{e_\lambda :\lambda\in\Lambda\}$ is an orthonormal basis for $L^2(\mu)$. A Lebesgue measurable subset $\Omega$ of $\br^d$ is called {\it spectral} if the renormalized Lebesgue measure on $\Omega$ is spectral. We say that $\Omega$ {\it tiles} $\br^d$ by translations if there exists a set $\mathcal T$ in $\br^d$ such that $\{\Omega+t : t\in\mathcal T\}$ is a partition of $\br^d$ (up to Lebesgue measure zero). 

A finite subset $A$ of $\br^d$ is called spectral if the measure $\frac{1}{|A|}\sum_{a\in A}\delta_a$ is spectral, where $\delta_a$ is the Dirac measure at $a$. 
\end{definition}

Fuglede's conjecture \cite{Fug74} asserts that a Lebesgue measurable subset $\Omega$ of $\br^d$ is spectral if and only if it tiles $\br^d$ by translations. Tao \cite{Tao04} found a union of cubes, in dimension 5 or higher, which is spectral but does not tile. Later, Tao's counterexample was improved by Matolcsi and his collaborators \cite{KM06,FaMaMo06}, to disprove Fuglede's conjecture in both directions, down to dimension 3. In dimension 1 and 2, the conjecture is still open in both directions. 

Lebesgue measure is not the only measure that provides examples of spectral sets. In \cite{JP98}, Jorgensen and Pedersen showed that the Hausdorff measure on a fractal Cantor set with scale 4 is also spectral and a spectrum has the form:
$$\Lambda=\{\sum_{k=0}^n 4^k l_k : l_k\in\{0,1\}, n\in\bn\},$$
but there are many more spectra for the same measure as shown in \cite{DHS09}. Many more examples of fractal spectral measures have been constructed since \cite{Str00,LaWa02,DuJo07b}.

Finite spectral sets of integers are closely tied \cite{Lab02} to a conjecture of Coven and Meyerowitz \cite{CoMe99} on translational tilings of $\bz$. 

In this paper we focus on the following question: which sets appear as spectra of some measure? The main result is a characterization of spectra of measures in terms of the existence of a strongly continuous representation of the ambient group which has a {\it wandering vector} for the given set. Wandering vectors are vectors that generate orthonormal bases under the action of some system of unitary operators. They are ubiquitous throughout mathematics \cite{HaLa00,HaLa01,Iz11,CPT11,Be05}.

\begin{definition}\label{defw1}
Let $\U$ be a family of unitary operators acting on a Hilbert space $\H$. We say that a vector $v_0\neq 0$ in $\H$ is a {\it wandering vector} if $\{Uv_0 : U\in\U\}$ is an orthogonal family of vectors. 
\end{definition}

We keep a higher level of generality and work with locally compact abelian groups.

\begin{definition}\label{def2.1}

Let $\Gamma$ be a locally compact abelian group and let $G$ be its dual group (of all continuous characters); we will write $G=\widehat \Gamma$ and $\widehat G=\widehat{\widehat\Gamma}\approx\Gamma$ where the isomorphism $\widehat{\widehat\Gamma}\approx\Gamma$ is the Pontryagin duality theorem; see \cite{Rud90}. For a point $\gamma\in\Gamma$, write 
\begin{equation}
\ip{\gamma}{g}=e_\gamma(g),\quad(g\in G).
\label{eq5.1}
\end{equation}

We say that a subset $S$ of $\Gamma$ is a {\it spectrum} for a Borel probability measure $\mu_0$ on $G$ if the set $\{e_\gamma :\gamma\in S\}$ is an orthonormal basis for $L^2(\mu_0)$. 
\end{definition}

Because of the interest in Fourier frames \cite{OSANN} and since it does not affect the simplicity of the statement of our theorem, we formulate it not just for orthormal bases of exponential functions but also for frames.
\begin{definition}\label{def1.2}
Let $A,B>0$ and let $\H$ be a Hilbert space a family of vectors $\{e_i : i\in I\}$ in $\H$ is called a {\it frame with bounds $A,B$} if 
$$A\|f\|^2\leq\sum_{i\in I}|\ip{e_i}{f}|^2\leq B\|f\|^2,\quad(f\in\H)$$

A subset $S$ of $\Gamma$ is a {\it frame spectrum} with bounds $A,B$ for a Borel probability measure $\mu_0$ on $G$ if the set $\{e_\gamma :\gamma\in S\}$ is a frame with bounds $A,B$ for $L^2(\mu_0)$. 
\end{definition}

\begin{theorem}\label{th5.1}
Let $S\subset\Gamma$ be an arbitrary subset. Then the subset $S$ is a spectrum/frame spectrum with bounds $A,B$ for a Borel probability measure $\mu_0$ on $G$ if and only if there exists a triple $(\H,v_0, U)$ where $\H$ is a complex Hilbert space, $v_0\in\H$, $\|v_0\|=1$ and $U(\cdot)$ is a strongly continuous representation of $\Gamma$ on $\H$ such that $\{U(\gamma)v_0 : \gamma\in S\}$ is an orthonormal basis/frame with bounds $A,B$ for $\H$. 

Moreover, in this case $\mu_0$ can be chosen such that
\begin{equation}
\ip{v_0}{U(\xi)v_0}_{\H}=\int_Ge_\xi(g)\,d\mu_0(g)\mbox{ for all }\xi\in \Gamma
\label{eq5.1.1}
\end{equation}
and there is an isometric isomorphism $W:L^2(G,\mu_0)\rightarrow\H$ such that 
\begin{equation}
We_\gamma=U(\gamma)v_0\mbox{ for all }\gamma\in\Gamma.
\label{eq5.1.2}
\end{equation}

\end{theorem}

In section 3 we illustrate how this result can be used to determine spectral sets of a particular form. Our focus is on the techniques, more than on the examples themselves. As a corollary, we describe all spectral sets with 3 elements.

\section{Proof of Theorem \ref{th5.1}}
For simplicity, we will prove Theorem \ref{th5.1} for orthonormal basis; for frames, the proof is identical, just replace the words ``orthonormal basis'' by the words ``frame with bounds $A,B$''.
\begin{proof}
Suppose $S$ is a spectrum for $\mu_0$. Set $\H=L^2(G,\mu_0)$, $v_0=$the constant function 1 in $L^2(G,\mu_0)$ and take, for $\xi\in\Gamma$, $U(\xi)$ on $L^2(G,\mu_0)$ to be the multiplication operator , i.e., 
\begin{equation}
(U(\xi)f)(g)=e_\xi(g)f(g)\quad(f\in L^2(G,\mu_0),\xi\in\Gamma,g\in G)
\label{eq5.1.5}
\end{equation}
A simple check shows that all the requirements are satisfied and the isomorphism $W$ is just the identity. 

Conversely, suppose $(\H,v_0,U)$ is a triple such that $\{U(\gamma)v_0 : \gamma\in S\}$ is orthonormal in $\H$. Then by the Stone-Naimark-Ambrose-Godement theorem (the SNAG theorem \cite{Ma91,Ma04}), there is an orthogonal projection valued measure $P_U$ defined on the Borel subsets of $G$, such that 

\begin{equation}
U(\xi)=\int_Ge_\xi(g)\,dP_U(g)\quad(\xi\in\Gamma)
\label{eq5.1.8}
\end{equation}
Now set 
\begin{equation}
d\mu_0(g):=\|dP_U(g)v_0\|_\H^2
\label{eq5.1.9}
\end{equation}
and note that $\mu_0$ will then be a Borel probability measure on $G$. 

We prove that \eqref{eq5.1.1} holds. 

Let $\xi\in\Gamma$. We have 
\begin{equation}
\int_Ge_\xi(g)\,d\mu_0(g)=\int_Ge_\xi(g)\|dP_U(g)v_0\|_\H^2=\int_Ge_\xi(g)\,\ip{v_0}{dP_U(g)v_0}
\label{eq5.1.10}
\end{equation}
$$=\ip{v_0}{\left(\int_Ge_\xi(g)\,dP_U(g)\right)v_0}=\ip{v_0}{U(\xi)v_0}.$$

We now show that there is an isometric isomorphism $W:L^2(G,\mu_0)\rightarrow\H$ that satisfies \eqref{eq5.1.2}. The fact that $\{e_\gamma :\gamma\in S\}$ is an orthonormal basis will follow from this. 
Define $We_\gamma=U(\gamma)v_0$ for $\gamma\in\Gamma$. We prove that the inner products are preserved by $W$ and this shows that $W$ can be extended to a well defined isometry from $L^2(G,\mu_0)$ onto $\H$; it is onto because $U(\gamma)v_0$ with $\gamma\in S$ is an orthonormal basis for $\H$, and it will be defined everywhere because the functions $e_\gamma$, $\gamma\in\Gamma$ are uniformly dense on any compact subset of $G$ so they are dense in $L^2(G,\mu_0)$. 
But according to \eqref{eq5.1.10}, we have for $\gamma,\gamma'\in\Gamma$:
$$\ip{U(\gamma)v_0}{U(\gamma')v_0}=\int_G\cj{e_\gamma(g)}e_{\gamma'}(g)\,d\mu_0(g).$$
\end{proof}

\begin{remark}\label{rem5.2}
In Theorem \ref{th5.1}, for $S$ to be the spectrum of some measure $\mu_0$, it is enough for the family $\{U(\gamma)v_0 : \gamma\in S\}$ to be an orthogonal basis for the closed linear span of $\H_{v_0}:=\{U(\xi)v_0 : \xi\in\Gamma\}$, because, in this case $\H_{v_0}$ is a reducing subspace for $U$, and one can restrict the representation to it. 
\end{remark}

\begin{remark}\label{rem2.1}
For $\psi$ continuous and compactly supported on $G$ we have
\begin{equation}
W\psi=\int_\Gamma\widehat\psi(\xi)U(\xi)v_0\,d\xi
\label{eq5.1.11}
\end{equation}
where $d\xi$ denotes the Haar measure on $\Gamma$ and $\widehat\psi$ denotes the Fourier transform, i.e., 
$$\widehat\psi(\xi)=\int_G\cj{e_\gamma (g)}\psi(g)\,dg,\mbox{ and by the inversion formula, } \psi(g)=\int_\Gamma\widehat\psi(\xi)e_\xi(g)\,d\xi,\quad(\xi\in\Gamma,g\in G)$$

Indeed, for $\gamma\in\Gamma$, we have 
$$\ip{\int_\Gamma\widehat\psi(\xi)U(\xi)v_0\,d\xi}{U(\gamma)v_0}_\H=\int_\Gamma\cj{\widehat\psi(\xi)}\ip{U(\xi)v_0}{U(\gamma)v_0}_\H\,d\xi=\mbox{ (by \eqref{eq5.1.10})}$$
$$=\int_\Gamma\cj{\widehat\psi(\xi)}\int_G\cj{e_\xi(g)}e_\gamma(g)\,d\mu_0(g)\,d\xi=\mbox{(by Fubini)}=\int_G e_\gamma(g)\int_\Gamma\cj{\widehat\psi(\xi)e_\xi(g)\,d\xi}\,d\mu_0(g)$$$$
=\mbox{(by the inversion formula)}=\int_Ge_\gamma(g)\cj{\psi(g)}\,d\mu_0(g)=\ip{\psi}{e_\gamma}_{L^2(G,\mu_0)}=\ip{W\psi}{We_\gamma}_\H=\ip{W\psi}{U(\gamma)v_0}_\H.$$
But, since $U(\gamma)v_0$, $\gamma\in G$ span the entire space $\H$, equation \eqref{eq5.1.11} follows.

\end{remark}

\begin{remark}\label{rem2.2}
Let $A$ be a finite spectral subset of $\br^d$. Theorem \ref{th5.1} and its proof shows that one can chose a strongly continuous one parameter group $(U(t))_{t\in\br^d}$ defined on $l^2(A)$, $U(t)$ being the multiplication by the function $e_t$ restricted to $A$. Thus the spectrum of $U(t)$ is $e^{2\pi i t\cdot a}$.

\end{remark}

\section{Examples}
In this section we show how our result can be used to determine spectral sets of a particular form. We urge the reader to focus more on the techniques than on the examples themselves, since we believe that these techniques can be applied to more general situations. 

For simplicity, we will introduce some notations and present a few techniques that we will use in this section. 

We will consider finite subsets $A$ of $\br$ which we assume to be spectral. Note that, for finite sets, being spectral and being the spectrum of some measure are equivalent notions. If $A$ has spectrum $B$ then $B$ has to be finite, $|A|=|B|$ and the matrix $\frac{1}{\sqrt{|A|}}(e^{2\pi i ab})_{a\in A, b\in B}$ is unitary. Hence $A$ is also a spectrum for $B$. Conversely, if $A$ is the spectrum of some measure $\mu_0$ then $\mu_0$ is atomic, supported on a finite set $B$, and the measures of the points in the support have to be equal (see e.g. \cite{DuLa12}) and therefore $A$ is a spectrum for $B$ and vice-versa.

 By Theorem \ref{th5.1}, there is a Hilbert space $\H$, a strongly continuous one-parameter unitary group $(U(t))_{t\in\br}$ on $\H$ and a vector $v_0\in\H$ such that $\{U(a)v_0 : a\in A\}$ is an orthonormal basis for $\H$. We will identify $a=v_a:=U(a)v_0$, $a\in A$. We will denote by $\{a_1,\dots,a_m\}$ the linear subspace generated by the vectors $a_1,\dots,a_m$. We also write $a$ for the subspace $\{a\}$.

We write $\{a_1,\dots,a_m\}\nar{t}\{b_1,\dots,b_n\}$ if the unitary $U(t)$ maps the subspace $\{a_1,\dots,a_m\}$ into the subspace $\{b_1,\dots, b_n\}$.

\begin{remark}\label{rem2.3}
We know that the vectors $\{a: a\in A\}$ form an orthonormal basis for $\H$. Suppose $A=\{a_1,\dots, a_m\}\cup \{a_1',\dots,a_n'\}$ and also $A=\{b_1,\dots,b_m\}\cup\{b_1',\dots,b_n'\}$, both disjoint unions, distinct elements. Suppose $\{a_1,\dots,a_m\}\nar{t}\{b_1,\dots,b_m\}$. Then $\{a_1',\dots,a_n'\}\nar{t}\{b_1',\dots,b_n'\}$ because $U(t)$ is unitary so it maps orthogonal subspaces to orthogonal subspaces.

Also, suppose we have $\{a_1,\dots,a_p\}\nar{t}\{c_1,\dots,c_k,d_1,\dots,d_l\}$ and in addition $\{a_1',\dots,a_k'\}\nar{t}\{c_1,\dots,c_k\}$, where $\{a_1,\dots,a_p,a_1',\dots,a_k'\}$ are all distinct. Then $\{a_1,\dots,a_p\}\nar{t}\{d_1,\dots,d_l\}$, because $U(t)$ is unitary so $\{a_1,\dots,a_p\}$ must be mapped into the orthogonal complement of $\{c_1,\dots,c_k\}$ in $\{c_1,\dots,c_k,d_1,\dots,d_l\}$.

\end{remark}

\begin{definition}\label{def2.4}
Let $\H=\oplus_{i=1}^nV_i$ be an orthogonal decomposition of the Hilbert space $\H$. We say that a unitary $U$ on $\H$ {\it permutes} the subspaces $\{V_i\}$ if there exists a permutation $\sigma$ of $\{1,\dots,n\}$ such that 
$UV_i=V_{\sigma(i)}$ for all $i=1,\dots,n$. We say that $U$ permutes {\it non-trivially} if the permutation $\sigma$ is not the identity.
\end{definition}

\begin{lemma}\label{lem2.5}
Let $(U(t))_{t\in\br}$ be a strongly continuous one-parameter unitary group on $\H$. Let $a,b\in\br$, $a,b\neq0$. Suppose $U(a)$ and $U(b)$ permute some subspaces $\oplus_{i=1}^n V_i=\H$, one of them non-trivially. Then $a/b$ is rational. 

\end{lemma}

\begin{proof}
Assume by contradiction that $a/b$ is irrational. Then the set $M:=\{ma+nb: m,n\in \bz\}$ is dense in $\br$. Also, since $U(ma+nb)=U(a)^mU(b)^n$  it follows that $U(ma+nb)$ also permutes the subspaces $V_i$. Let $\sigma_t$ be the permutation associated to $U(t)$ for $t\in M$. We show that every $U(t)$ permutes the subspaces $V_i$. 

Let $t\in\br$. Approximate $t$ by a sequence $t_n$ in $M$. Pick $v\in V_1$ with $\|v\|=1$. We have $U(t_n)v\in V_{\sigma_{t_n}(1)}$ and $U(t_n)v\rightarrow U(t)v$. This implies that $U(t_n)v$ are close together for $n$ large. But then $\sigma_{t_n}(1)$ and $\sigma_{t_m}(1)$ are equal for $n,m$ large because, otherwise $U(t_n)v\in V_{\sigma_{t_n}(1)}$ and $U(t_n)v\rightarrow U(t)v$ would lie in orthogonal subspaces and therefore the distance between them would be $\sqrt2$ by Pythagora's theorem. Consequently, $U(t)v$ must lie in the same subspace as $U(t_n)v$ for $n$ large. Varying $v$, we see that $U(t)$ permutes the subspaces $\{V_i\}$. This argument shows also that $\sigma_t$ and $\sigma_s$ are identical if $t$ is close to $s$. But then the function $t\mapsto \sigma_t$ is locally constant. Since $\br$ is connected, this means that $\sigma_t$ is constant. But $\sigma_0$ is the identity and one of $\sigma_a$ or $\sigma_b$ is not. The contradiction implies that $a/b$ is rational.
\end{proof}

\begin{lemma}\label{lem2.6}
Let $S$ be a subset of $\br$. Assume in addition that there exists $\alpha>0$ such that $S\subset\alpha\bz$. Suppose there exists a unitary $U$ on a Hilbert space $\H$ and a vector $v_0\in\H$ such that $\{U^sv_0 : \alpha s\in S\}$ is an orthonormal basis for $v_0$. 
Then there exists a measure $\mu_0$ on $[0,\frac1\alpha)$ such that $S$ is a spectrum for $\mu_0$. 
\end{lemma}

\begin{proof}
The dual of the group $\alpha\bz$ is $\bt=[0,\frac1\alpha)$ with addition modulo $\frac1\alpha$; the duality pairing is $\ip{t}{k\alpha}=e^{2\pi itk\alpha}$ for $t\in[0,\frac1\alpha),k\in\bz$. $U$ gives a representation of $\alpha\bz$ on $\H$ by $U(\alpha k)=U^k$. According to Theorem \ref{th5.1}, there is a measure $\mu_0$ on $[0,\frac1a)$ such that $S$ is a spectrum for $\mu_0$, in this duality. But the characters on $[0,\frac1a)$ are just restrictions of characters on $\br$ (which are the exponential functions $e_t$); therefore the result follows.
\end{proof}

\begin{proposition}\label{pr2.6}
Let $A=\{0,1,\dots,n-2,a\}$ with $n\in\bn$, $n\geq3$ and $a\in\br$, $a\neq0,1,\dots,n-2$. Then $A$ is spectral if and only if $a$ is rational and in its reduced form $a=p/q$, with $(p+q)\equiv0\mod n$.
\end{proposition}

\begin{proof}
Assume that $A$ is spectral and use the notation described above. We have $0\nar{1}1, 1\nar{1}2,\dots,n-3\nar{1}n-2$. With the rules in Remark \ref{rem2.3}, we obtain that $a\nar{1}\{0,a\}$. Then $1\nar{a-1}a\nar{1}\{0,a\}$ so $1\nar{a} \{0,a\}$. But since $0\nar{a}a$ we get with Remark \ref{rem2.3} that $1\nar{a}0$. Then $a\nar{1-a}1\nar{a}0$ so $a\nar{1}0$. Then only possibility for $n-2$ is $n-2\nar{1}a$. 

Hence, if we denote by $v_{n-1}:=v_a$ we have that $v_k\nar{1}v_{(k+1)\mod n}$ and $v_k\nar{m}v_{(k+m)\mod n}$ for $m\in\bz$ and $k\in\{0,\dots,n-1\}$. So $U(m)$ permutes cyclically the one-dimensional subspaces generated by $v_k$, $k=0,\dots,n-1$. 

Next, we compute how $U(a)$ acts on these subspaces. We $0\nar{a}a$. For $k\in\{0,\dots,n-2\}$ we have 
$k\nar{a-k}a\nar{k}(k+n-1)\mod n=k-1$ so $k\nar{a}k-1$. The only remaining possibility for $a$ is $a\nar{a} n-1$. Thus $v_k\nar{a}v_{(k-1)\mod n}$, for $k\in\{0,\dots,n-1\}$. 

So $U(a)$ permutes cyclically the subspaces $v_k$. With Lemma \ref{lem2.5} we see that $a$ has to be rational, $a=p/q$, irreducible, for some $p,q\in\bz$. Then $qa=p$ so $U(qa)=U(p)$. Then, apply this operator to the subspace of $v_0$, since 
$v_0\nar{qa} v_{-q\mod n}$ and $v_0\nar{p} v_{p\mod n}$ we get that $-q\equiv p\mod n$ so $p+q\equiv 0\mod n$. 

For the converse, assume $a$ has the given form. We will define a unitary operator $U(\frac1q)$ on $l^2(\{0,1,\dots,n-1\})$ as in Lemma \ref{lem2.6}. Since $p$ and $q$ are relatively prime, there exist $k,l\in\bz$ such that $kp+lq=1$. 
Let $\delta_i$ be the canonical basis in $l^2(\{0,\dots,n-1\})$. Define $U(\frac1q)\delta_i=\delta_{(i+l-k)\mod n}$ for $i=0,\dots,n-1$. Then define $U(\frac jq)=U(\frac1p)^j$ for $j\in\bz$. 

We have 
$$U(1)\delta_i=U(\frac1q)^q\delta_i=\delta_{(i+ql-qk)\mod n}=\delta_{(i+1-k(p+q))\mod n}=\delta_{(i+1)\mod n}.$$
Also
$$U(a)\delta_i=U(\frac1q)^p\delta_i=\delta_{(i+pl-pk)\mod n}=\delta_{(i-1+l(p+q))\mod n}=\delta_{(i-1)\mod n}.$$

Then, we see that $\{U(0)\delta_0,\dots,U(n-2)\delta_0,U(a)\delta_0\}=\{\delta_0,\dots,\delta_{n-2},\delta_{n-1}\}$ is an orthonormal basis. With Lemma \ref{lem2.6} we get that $A$ is spectral.

\end{proof}

We consider now spectral sets $A$ with 3 elements. The spectral property is invariant under translations and scaling. Therefore, using a translation we can assume that 0 is in $A$ and then, rescaling we can assume that also $1$ is in $A$. 
\begin{corollary}\label{cor2.7}
Consider a set $A$ with 3 elements $A=\{0,1,a\}$. Then $a$ is spectral if and only if $a$ is rational and if $a=p/q$ in its reduced form then $p+q$ is divisible by 3.
\end{corollary}

\begin{acknowledgements}
This work was done while the first named author (PJ) was visiting the University of Central Florida. We are grateful to the UCF-Math Department for hospitality and support. The authors are pleased to thank Professors Deguang Han and Qiyu Sun for helpful conversations. PJ was supported in part by the National Science Foundation, via a Univ of Iowa VIGRE grant. This work was partially supported by a grant from the Simons Foundation (\#228539 to Dorin Dutkay).
\end{acknowledgements}
\bibliographystyle{alpha}
\bibliography{eframes}

\end{document}